\documentclass[onefignum,onetabnum]{siamonline220329}

\usepackage{amssymb}
\usepackage{graphicx}
\usepackage{tikz}
\usetikzlibrary{arrows, calc, matrix, shapes.misc}
\usepackage{verbatim}
\ifpdf
  \DeclareGraphicsExtensions{.eps,.pdf,.png,.jpg}
\else
  \DeclareGraphicsExtensions{.eps}
\fi

\newsiamthm{example}{Example}

\usepackage{enumitem}
\setlist[enumerate]{leftmargin=.5in}
\setlist[itemize]{leftmargin=.5in}

\newsiamremark{remark}{Remark}
\newsiamremark{hypothesis}{Hypothesis}
\crefname{hypothesis}{Hypothesis}{Hypotheses}
\newsiamthm{claim}{Claim}

\headers{Variance of the distance to the boundary of convex domains in $\mathbb{R}^{2}$ and $\mathbb{R}^{3}$}{A. N. Fletcher and A. G. Fletcher}

\title{Variance of the distance to the boundary of convex domains in $\mathbb{R}^{2}$ and $\mathbb{R}^{3}$\thanks{Submitted to the editors 1st July 2024.
\funding{This work was funded by grants from the Simons Foundation (352034), the Biotechnology and Biological Sciences Research Council (BB/V018647/1 and BB/Y514020/1), and the Engineering and Physical Sciences Research Council (EP/W024144/1).}}}

\author{Alastair N. Fletcher\thanks{Department of Mathematical Sciences, Northern Illinois University, Dekalb, IL 60115, USA 
  (\email{afletcher@niu.edu}).}
\and Alexander G. Fletcher\thanks{School of Mathematics and Statistics, University of Sheffield, Hicks Building, Hounsfield Road, Sheffield S3 7RH, UK 
  (\email{a.g.fletcher@sheffield.ac.uk}).}}

\begin{document}

\maketitle

\begin{abstract}
In this paper, we give for the first time a systematic study of the variance of the distance to the boundary for arbitrary bounded convex domains in $\mathbb{R}^2$ and $\mathbb{R}^3$. 
In dimension two, we show that this function is strictly convex, which leads to a new notion of the centre of such a domain, called the variocentre. 
In dimension three, we investigate the relationship between the variance and the distance to the boundary, which mathematically justifies claims made for a recently developed algorithm for classifying interior and exterior points with applications in biology.
\end{abstract}



\section{Introduction} \label{sec:Introduction}

There are many different notions of the centre of a convex domain. For instance, triangles have the circumcentre, the centroid, the incentre, the orthocentre, the nine-point centre and the Chebyshev centre. For general polygons, the usual notion of centre is the centroid, or centre of mass, but other notions of centre can be defined by placing equal masses at the vertices, called the vertex centroid, or by assuming the sides of the polygon have equal mass per unit length, called the side centroid.

In this paper, we investigate another notion of the centre of a convex domain that has received recent attention. If $U \subset \mathbb{R}^{n}$, for $n\geq 2$, is a bounded convex domain and if $x_0 \in U$, then for any direction $\sigma \in S^{n-1}$, there is a unique point on $\partial U$ in that direction from $x_0$. This naturally gives a distance-to-the-boundary function $S^{n-1}\to \mathbb{R}^{+}$ that depends on $x_0$. We can ask for the various integral means of this function that we will denote by $I_k^{U}(x_0)$ (we will make this more precise later). Observe that the usual distance to the boundary $d(x_0, \partial U)$ is just the minimum of this function over $\sigma \in S^{n-1}$.

The quantity that is central to our investigations is the variance of the distance to the boundary function given by $v_U(x_0) = I_2^{U}(x_0) - [ I_1^{U}(x_0) ]^2$. It is clear that if $U$ is a ball and $x_0$ is its usual centre, then the distance to the boundary in any direction is the same, from which it follows that the variance $v_U(x_0) = 0$. For any other point of the ball, elementary considerations show that the variance is strictly positive. Therefore the variance is minimized at the centre of the ball. It is this idea that we wish to explore for more general convex domains. First, we investigate the regularity properties of the variance function and second, we wish to ascertain whether or not the variance function has a unique minimum.

The distance to the boundary function $d(\cdot, \partial U)$ is well-known to be concave, with regularity close to the boundary depending on the regularity of $\partial U$, see for example~\cite[Lemma 14.16]{GT01}. The interplay between the distance to the boundary function and the variance function is a large factor in the motivation of our study.

S. Strawbridge et al~\cite{Straw23} made use of both of these notions in an unsupervised machine learning algorithm for accurately and robustly classifying interior and exterior points of a 3D point-cloud. 
In particular, it was stated in~\cite{Straw23} that if $U$ is the unit ball in $\mathbb{R}^{3}$, then $v_{U}(x_{0})$ has a unique minimum at the centre of the ball. 
Moreover, a key part of the algorithm studied in~\cite{Straw23} is an inverse relationship between the distance to the boundary and the variance. 
This relationship was motivated based on empirical investigations; one could come up with alternative statistics allowing clustering of interior and exterior points. 
The classification algorithm developed in~\cite{Straw23} is useful in a variety of biological settings where cells located on the exterior of a tissue behave differently to those on its interior. For example, in the early mouse embryo (a common experimental proxy for human embryonic development), exterior cells give rise to the placenta, while interior cells give rise to the foetus and yolk sac. 
In cases where it is not possible to reliably directly label and image specific proteins to reveal which cells will give rise to which tissues, a classification algorithm based solely on cells' spatial locations is required.

Inspired by this work, the first named author with K. Fletcher and J. Wasiqi~\cite{FFW23} conducted a mathematical study of $v_{T}(x_{0})$ for planar triangles $T$.
It was shown in this paper that if $\gamma$ is a line segment joining a vertex of a triangle to a point on the opposite side, then the variance restricted to $\gamma$ is a convex function. However, this appears to be some way short of showing that the variance function as a whole is convex.

In this paper, we will address this question and, in fact, resolve the question of convexity in dimension $2$.

\begin{theorem}
    \label{thm:dim2convex}
Let $U$ be a bounded convex domain in $\mathbb{R}^2$. Then $v_U$ is a strictly convex function.
\end{theorem}

As $v_U$ is non-negative, it immediately follows that $v_U$ has a unique minimum that we call the variocentre. We leave the question of how the variocentre relates to other notions of centre to future work. Returning to our motivation from~\cite{Straw23}, we will see that the distance to the boundary and the variance need not always be inversely related, see Example \ref{ex:2d} below. However, once we are sufficiently close to the boundary, this relationship does hold, as the following result shows.

\begin{theorem}
\label{thm:vvskstatement}
Let $U$ be a bounded convex domain in $\mathbb{R}^{2}$, let $z_{0} \in U$, let $r= d(z_{0} , \partial U)$ and suppose that $e^{i\sigma} \in S^1$ is such that $w:= z_{0} + re^{i\sigma} \in \partial U$. For $\delta \in (0,1)$, set $z_{\delta} = (1-\delta)w + \delta z_{0}$. Then
\[ D_{\sigma} v_U(z_{\delta}) = O \left( \ln \frac{1}{\delta} \right) \]
as $\delta \to 0$.
\end{theorem}

Following on from the example of the ball described above, we may also show that if a convex domain has a boundary that is close to circular, then this inverse relationship between the distance to the boundary and the variance holds once we are outside a compact set.

\begin{theorem}
\label{thm:ellipticstatement}
There exists $\epsilon_{0}>0$, such that if $0< \epsilon < \epsilon _{0}$ and $U$ is a bounded convex domain for which $\partial U$ is contained in $\{z : 1\leq |z| \leq 1+\epsilon \}$, then there exists $r(\epsilon)$ such that if $r(\epsilon) < r <1$ and $|z| = r$, then
\[ D_{\arg z} v_U (z) > 0.\]
Moreover, $r(\epsilon) \to 0$ as $\epsilon \to 0$.
\end{theorem}

We also study variance in dimension three, although our results are somewhat less complete. In particular, we do not obtain convexity of $v_U$ in this setting, although we expect this to still be true. We do obtain the following analogues of Theorem \ref{thm:vvskstatement} and Theorem \ref{thm:ellipticstatement}.

\begin{theorem}
\label{thm:vvsk3dstatement}
Let $U$ be a bounded convex domain in $\mathbb{R}^{3}$, let $x_{0} \in U$, let $r= d(x_{0} , \partial U)$ and suppose that $\sigma \in S^{2}$ is such that $w:= x_{0} + r\sigma \in \partial U$. For $\delta \in (0,1)$, set $x_{\delta} = (1-\delta)w + \delta x_{0}$. Then
\[ D_{\sigma} v_U(x_{\delta}) = O \left( \ln \frac{1}{\delta} \right) \]
as $\delta \to 0$.
\end{theorem}

\begin{theorem}
\label{thm:elliptic3dstatement}
There exists $\epsilon_{0}>0$ such that if $0< \epsilon < \epsilon _{0}$ and $U\subset \mathbb{R}^{3}$ is a bounded convex domain for which $\partial U$ is contained in $\{x\in \mathbb{R}^{3} : 1\leq |x| \leq 1+\epsilon \}$, then there exists $r(\epsilon)$ such that if $r(\epsilon) < r <1$ and $|x| = r$, then
\[ D_{x/|x|} v_U (x) > 0.\]
\end{theorem}

The paper is organized as follows. 
In Section~\ref{sec:Preliminaries} we properly define the variance function, give some of its basic properties, as well as recall relevant definitions and properties of elliptic integrals that will be needed in the dimension $2$ analysis.
In Section~\ref{sec:Dimension_2} we focus on dimension $2$ and provide the proofs of Theorem \ref{thm:dim2convex}, Theorem \ref{thm:vvskstatement} and Theorem \ref{thm:ellipticstatement}.
Finally, in Section~\ref{sec:Dimension_3} we move to dimension $3$ and provide the proofs for Theorem \ref{thm:vvsk3dstatement} and Theorem \ref{thm:elliptic3dstatement}.

\section{Preliminaries} \label{sec:Preliminaries}

\subsection{Distance to the boundary}

We begin by formalizing the definitions outlined at the start of Section~\ref{sec:Introduction}. Let $n\geq 2$ and let $U \subset \mathbb{R}^{n}$ be a non-empty bounded convex domain. In this paper we will focus on dimensions $2$ and $3$, but our notions can be applied in any dimension.

\begin{definition}
Let $U \subset \mathbb{R}^{n}$ be a bounded convex domain.
\begin{enumerate}[label=(\alph*)]
\item For $x_{0} \in U$ and $\sigma \in S^{n-1}$, denote by $d_U(x_{0},\sigma)$ the distance from $x_{0}$ to $\partial U$ in the direction $\sigma$. More precisely, if $\gamma$ is the intersection of $\{ v = x_{0} + t\sigma : t \geq 0 \}$ with $U$, then $d_U(x_{0} , \sigma)$ is the length of $\gamma$. 
\item For $x_{0} \in U$ and $k\in \mathbb{N}$, we define
\[ I^{U}_k(x_{0}) = \frac{1}{A(S^{n-1})} \int_{S^{n-1}} d_U(x_{0} , \sigma )^k \: dA,\]
where $dA$ denotes a spherical volume element of $S^{n-1}$ and $A(S^{n-1})$ is the $(n-1)$-dimensional volume of the unit sphere in $\mathbb{R}^{n}$.
\item We define the variance of the distance to the boundary of $U$ from $x_{0}\in U$ to be
\[ v_U(x_{0}) = I_{2}^{U}(x_{0}) - [ I_{1}^{U}(x_{0}) ] ^{2}.\]
\end{enumerate}
\end{definition}

If the context is clear, we may suppress the $U$ notation and just write $d(x_{0},\sigma), I_k(x_{0}), v(x_{0})$ etc. While we initially define $v_U$ on $U$, it may be extended to $\partial U$ either by continuity, or using the definition above and allowing the distance from $x_{0}$ to $\partial U$ in certain directions to be zero. Observe that $I_{1}(x_{0})$ is the average distance to the boundary of $U$ from $x_{0}$. It is elementary to check that if $U$ is a ball and $x_{0}$ is its centre, then $I_{1}(x_{0})$ is the radius of the ball, and $v(x_{0})$ is zero.

We start by giving some general properties of the variance.

\begin{theorem}
\label{thm:props}
Let $U$ be a bounded convex domain in $\mathbb{R}^{n}$. 
\begin{enumerate}[label=(\alph*)]
\item If $f:U \to \mathbb{R}^{n}$ is an isometry, then 
\[ v_{f(U)} ( f(x) ) = v_U (x).\]
\item If $f:U \to \mathbb{R}^{n}$ is a similarity with scaling factor $C>0$, then 
\[ v_{f(U)} ( f(x) ) = C^{2} v_U (x).\]
\end{enumerate}
\end{theorem}

These were proved in dimension $2$ in \cite{FFW23}, but for the convenience of the reader, we prove these in the general setting here. 

\begin{proof}
Let $k\in \mathbb{N}$ and let $f$ be an isometry. Then, by an elementary change of variables, we have
\begin{align*} 
I^{f(U)}_k(f(x)) &=  \frac{1}{A(S^{n-1})} \int_{S^{n-1}} d_{f(U)}(f(x) , \sigma )^k \: dA \\
&=  \frac{1}{A(S^{n-1})} \int_{S^{n-1}} d_{U}(x , \sigma )^k \: dA \\
&= I^{U}_k(x).
\end{align*}
Combining this with the formula for $v_U$, this proves (a).

If $f$ is a dilation by factor $C>0$, we may assume $x=0$ by conjugating by a translation and using part (a). Then, for $k\in \mathbb{N}$,
\begin{align*}
I^{f(U)}_k(0) &=  \frac{1}{A(S^{n-1})} \int_{S^{n-1}} d_{f(U)}(0 , \sigma )^k \: dA\\
&=  \frac{C^k}{A(S^{n-1})} \int_{S^{n-1}} d_{U}(0 , \sigma )^k \: dA \\
&= C^k I^{U}_k(0).
\end{align*}
Combining this with the formula for $v_U$, this proves (b) in this case.
For the general case, by decomposing the similarity as an isometry followed by a dilation and applying the above computation and part (a), we obtain (b).
\end{proof}

Recall that the Hausdorff metric between two sets $U,V$ in $\mathbb{R}^{n}$ is given by
\[ d_H(U,V) = \max \left \{ \sup_{x\in U} d(x, V) , \: \sup_{y\in V} d(y,U) \right \} ,\]
where $d(x,V) = \inf_{y\in V} d(x,y)$. Using this metric gives us a way to talk about the convergence of sets: we say that $U_m$ converges to $U_{0}$ if $d_H(U_m,U_{0}) \to 0$ as $m\to \infty$. If we assume all our sets are bounded convex domains, then we may equivalently state that given $\epsilon >0$, we can find $M\in \mathbb{N}$ so that $\partial U_m$ is contained in the $\epsilon$-neighbourhood $N_{\epsilon}(\partial U_{0})$ of $\partial U_{0}$ for $m\geq M$ (see Figure~\ref{fig:1}).

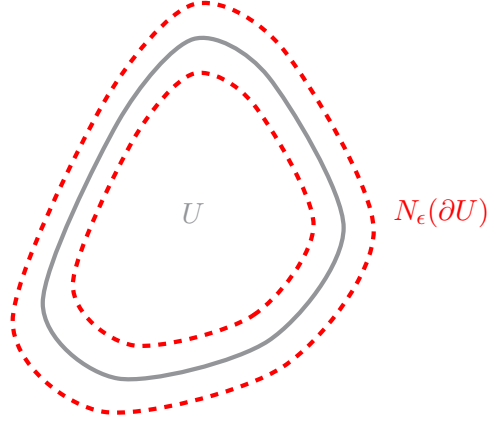
\begin{figure}[ht]
\begin{center}
\begin{tikzpicture}
\draw [gray, ultra thick] plot [smooth cycle] coordinates {(-1,-2.2) (1,-1.7) (2,-0.2) (1,1.8) (0,2.3) (-1,1.2) (-2,-1.2)};
\begin{scope}[scale=1.2]
\draw [red, ultra thick, dashed] plot [smooth cycle] coordinates {(-1,-2.2) (1,-1.7) (2,-0.2) (1,1.8) (0,2.3) (-1,1.2) (-2,-1.2)};
\end{scope}
\begin{scope}[scale=0.8]
\draw [red, ultra thick, dashed] plot [smooth cycle] coordinates {(-1,-2.2) (1,-1.7) (2,-0.2) (1,1.8) (0,2.3) (-1,1.2) (-2,-1.2)};
\end{scope}
\node[text=gray] at (0,0) {$U$};
\node[text=red] at (3.3,0) {$N_{\epsilon}(\partial U)$};
\end{tikzpicture}
\caption{An $\epsilon$-neighbourhood of $\partial U$.}
\label{fig:1}
\end{center}
\end{figure}

\begin{theorem}
\label{thm:converge}
Let $(U_m)_{m=1}^{\infty}$ and $U_{0}$ be bounded convex domains with $d_H(U_m,U_{0}) \to 0$ as $m\to \infty$. Then for any $x_{0} \in U$, $\sigma \in S^{n-1}$ and $k\in \mathbb{N}$,
\[ | d_{U_m}(x_{0} , \sigma) - d_{U_{0}} ( x_{0} , \sigma) | \to 0,\]
\[ | I_k^{U_m}(x_{0}) - I_k^{U_{0}}(x_{0}) | \to 0,\]
and
\[ |v_{U_m}(x_{0}) - v_{U_{0}}(x_{0}) | \to 0\]
as $m\to \infty$.
\end{theorem}

Note that if $x_{0} \in U_{0}$, then it may not be the case that $x_{0} \in U_m$ for all $m$, but certainly for all large enough $m$ this will be the case and so the conclusion in this result makes sense.

\begin{proof}
Fix $x_{0} \in U_{0}$ and $\sigma \in S^{n-1}$. Let $\delta = d(x_{0} , \partial U_{0}) >0$. We may choose $M\in \mathbb{N}$ large enough so that $\partial U_m$ is contained in $N_{\delta /2}(\partial U_{0})$ for $m\geq M$. We may then define $y_{0}$ and $y_m$, for $m\geq M$, to be the unique point on $\partial U_{0}$ and $\partial U_m$, respectively, from $x_{0}$ in the direction $\sigma$.

Now, given $\epsilon \in (0,\delta/2)$, we may find $M' \geq M$ such that $\partial U_m$ is contained $N_{\epsilon}(\partial U_{0})$ for $m\geq M'$. It follows that $|y_{0} - y_m| < \epsilon$ for $m\geq M'$. We conclude that $y_m \to y_{0}$ and hence 
\[| d_{U_m}(x_{0} , \sigma) - d_{U_{0}} ( x_{0} , \sigma) | \to 0\]
as $m\to \infty$.
Observe that this convergence is uniform over $\sigma \in S^{n-1}$. We therefore have, for $k\in \mathbb{N}$ and $m\geq M'$, 
\begin{align*}
| I_k^{U_m}(x_{0}) - I_k^{U_{0}}(x_{0}) | &= \frac{1}{A(S^{n-1})} \left | \int_{S^{n-1}} \left( d_{U_m}(x_{0} , \sigma)^k - d_{U_{0}}(x_{0} , \sigma) ^k \right) \: dA \right | \\
&\leq \sup_{\sigma \in S^{n-1} } \left | d_{U_m}(x_{0} , \sigma) - d_{U_{0}}(x_{0} , \sigma)  \right | \cdot \left( \sum_{ i = 0}^{k-1} d_{U_m}(x_{0} , \sigma)^i d_{U_{0}}(x_{0},\sigma)^{k-1-i} \right) \\
& \leq \epsilon \cdot k(\operatorname{diam} U_{0} + 2\epsilon)^{k-1}.
\end{align*}
It follows that
\[ | I_k^{U_m}(x_{0}) - I_k^{U_{0}}(x_{0}) | \to 0\]
as $m\to \infty$. It then follows from the formula for variance that
\[ |v_{U_m}(x_{0}) - v_{U_{0}}(x_{0}) | \to 0\]
as $m\to \infty$.
\end{proof}

Key to our study is to understand regularity properties of the variance. It may not be the case that $v_U$ is differentiable as a function $U \to \mathbb{R}^{+}$, but it will be profitable to study the directional derivatives. Recall that if $\sigma \in S^{n-1}$, then the directional derivative of $v_U$ in the direction $\sigma$, if it exists, is given by
\[ D_{\sigma} v _U(x_{0}) = \lim_{t\to 0} \frac{ v_U(x_{0} + t\sigma) - v_U(x_{0}) } {t} .\]

\subsection{Elliptic integrals}

We recall here some material on elliptic integrals that will be important in the sequel. Let $0<k<1$. The complete elliptic integral of the first kind is defined by
\begin{equation}
\label{eq:ellK} 
K(k) = \int_{0}^{\pi/2} \frac{d\theta}{\sqrt{1-k^{2}\sin^{2}\theta}} = \int_{0}^1 \frac{ dx}{ \sqrt{ (1-x^{2})(1-k^{2}x^{2}) } } .
\end{equation}
This can be expressed as the following power series:
\[ K(k)  = \frac{\pi}{2} \sum_{n=0}^{\infty} \left( \frac{ (2n)! }{2^{2n} (n!)^{2} } \right)^{2} k^{2n}.\]
Evidently the integral expressions for $K(k)$ diverge as $k\to 1^-$. However, we have the asymptotic behaviour
\begin{equation}
\label{eq:K1} 
K(k) = -\tfrac12 \ln(1-k) ( 1+ O(1-k) ) + \ln 4 + O(1-k)
\end{equation}
as $k\to 1^-$.

The complete elliptic integral of the second kind is defined by 
\begin{equation}
\label{eq:ellE} 
E(k) = \int_{0}^{\pi/2} \sqrt{1-k^{2}\sin^{2}\theta} \: d\theta = \int_{0}^1 \frac{ \sqrt{1-k^{2}x^{2}}}{\sqrt{1-x^{2}}  } \: dx .
\end{equation}
It is well-known that $E$ is a decreasing concave function in $k$.
We have the following power series representation for $E$:
\[ E(k) = \frac{\pi}{2} \sum_{n=0}^{\infty} \left( \frac{ (2n)! }{2^{2n} (n!)^{2} } \right)^{2} \frac{k^{2n}}{1-2n} .\]
It is clear that the integral expressions for $E(k)$ converge as $k\to 1^-$, and we in fact have the asymptotic behaviour
\begin{equation}
\label{eq:E1} 
E(k) = 1 + O(1-k)
\end{equation}
as $k\to 1^-$.

The difference between these elliptic integrals can be given concisely by
\begin{equation}
\label{eq:ellKE} 
K(k) - E(k) = \pi \sum_{n=1}^{\infty} \left( \frac{ (2n-1)!! }{(2n)!! } \right)^{2} k^{2n},
\end{equation}
where the double factorial is defined by
\[ n!! = \prod_{m=0}^{ \lceil n/2 \rceil -1 } (n-2m).\]

\section{Dimension \texorpdfstring{$2$}{2}} \label{sec:Dimension_2}

In dimension two, we will take advantage of complex coordinates. 
Moreover, we will employ an abuse of notation and write $D_{\sigma}$ for $D_{e^{i\sigma}}$ where $\sigma \in [0,2\pi)$.
First, there is the striking result that $I_{2}^{U}(x_{0})$ is constant.

\begin{theorem}[\cite{FFW23}, Theorem 1.2]
\label{thm:i2}
For any bounded convex domain $U\subset \mathbb{R}^{2}$ and any $x_{0}\in U$ we have
\[ I_{2}^{U}(x_{0}) = \frac{1}{\pi}  \operatorname{Area}(U).\]
\end{theorem}

This is a quick application of Green's Theorem, that we recall for the convenience of the reader.

\begin{proof}
Suppose first that $\partial U$ is smooth and that we parameterize $\partial U$ by fixing $z \in \overline{U}$ and setting $\gamma( \phi) = d_U(z,\phi ) e^{i\phi}$ for $0 \leq \phi \leq 2\pi$. Then $\gamma$ is a smooth function and 
\[ \gamma ' (\phi) = (d_U'(z,\phi ) + id_U(z,\phi ) ) e^{i\phi}, \]
where $d_U'$ denotes the derivative with respect to $\phi$.
As $d_U(z,\phi )$ is real-valued, we have
\begin{align*}
\int_{\gamma} \overline{z} \: dz &= \int_{0}^{2\pi} \overline {\left( d_U(z,\phi )e^{i\phi} \right)} ( d_U'(z,\phi ) + id_U(z,\phi) ) e^{i\phi} )\: d\phi \\
&= \int_{0}^{2\pi} ( d_U(z,\phi) d_U'(z,\phi )+ i(d_U(z,\phi ))^{2} ) \: d\phi.
\end{align*}
Integrating by parts, we have
\begin{align*} 
\int_{0}^{2\pi} d_U(z,\phi ) d_U'(z,\phi ) \: d\phi &= \left [ d_U(z,\phi )^{2} \right ] _{\phi= 0}^{2\pi} - \int_{0}^{2\pi}d_U(z,\phi ) d_U'(z,\phi )\: d\phi \\
&= 0 - \int_{0}^{2\pi} d_U(z,\phi ) d_U'(z,\phi ) \: d\phi.
\end{align*}
Thus 
\[ \int_{0}^{2\pi} d_U(z,\phi ) d_U'(z,\phi ) \: d\phi = 0\]
and we have
\[ \int_{\gamma} \overline{z} \: dz = i\int_{0}^{2\pi} d_U(z,\phi )^{2} \: d\phi = 2\pi i I_{2}^{U}(z).\]
By Green's Theorem, $\int_{\gamma} \overline{z} \: dz$ is nothing other than $2i$ multiplied by the area enclosed by $\gamma$, from which Theorem~\ref{thm:i2} follows in the smooth case. For the general case, we may approximate $\partial U$ by smooth curves, apply the above argument and then appeal to Theorem~\ref{thm:converge}.
\end{proof}

We next give perhaps the simplest example of computing the variance.

\begin{example}
\label{ex:disk}
Let $U$ be the unit disk $\mathbb{D}$. Then for any $z\in \mathbb{D}$ we have, by Theorem~\ref{thm:i2}, that $I_{2}^{\mathbb{D}}(z)  =1$. If $z=0$, then $I_{1}^{\mathbb{D}}(0)$ is clearly equal to $1$ too, and so $v_{\mathbb{D}}(0) = 0$. For any other $z\in \mathbb{D}\setminus \{ 0 \}$, by rotational symmetry we have $v_{\mathbb{D}}(z) = v_{\mathbb{D}}(|z|)$ and so we may assume that $z=x\in (0,1)$.

To find $d_{\mathbb{D}}(x,\phi)$, we need to solve $|x+re^{i\phi}| = 1$ for $r$ (see Figure~\ref{fig:2}).
This gives 
\[ r^{2} + 2xr\cos\phi + x^{2}-1 = 0.\]
Solving the quadratic and taking the appropriate square root, we obtain
\[ r = -x\cos\phi + (1-x^{2}\sin^{2}\phi)^{1/2}.\]
Therefore
\begin{align*}
I_{1}^{U}(x) &= \frac{1}{2\pi} \int_{0}^{2\pi} \left( -x\cos \phi + (1-x^{2}\sin^{2}\phi)^{1/2} \right) \: d\phi \\
&= \frac{1}{2\pi} \int_{0}^{2\pi} (1-x^{2}\sin^{2}\phi)^{1/2} \: d\phi \\
&= \frac{2}{\pi} \int_{0}^{\pi/2} (1-x^{2}\sin^{2}\phi)^{1/2} \: d\phi \\
&= \frac{2 E(x)}{\pi},
\end{align*}
by \eqref{eq:ellE}. We conclude that
\[ v_{\mathbb{D}}(z) = 1 - \frac{ 4E(|z|)^{2}}{\pi^{2}}.\]
\end{example}

\begin{figure}[ht]
\begin{center}
\begin{tikzpicture}
\draw[gray, ultra thick](0,0) circle (3);
\draw[dashed, gray, thick] (0,0) -- (3,0);
\filldraw [gray] (0,0) circle (2pt);
\node[text=gray] at (0,-0.5) {$0$};
\filldraw [gray] (1.5,0) circle (2pt);
\node[text=gray] at (1.5,-0.5) {$x$};
\draw[red, ultra thick] (1.5,0) -- (2.7,1.307669);
\node[text=red] at (2,1.1) {$r$};
\draw[red, ultra thick] (2.1,0) arc (0:30:1);
\node[text=red] at (2.5,0.4) {$\phi$};
\end{tikzpicture}
\caption{Distance to the boundary from $x>0$ in $\mathbb{D}$.}
\label{fig:2}
\end{center}
\end{figure}
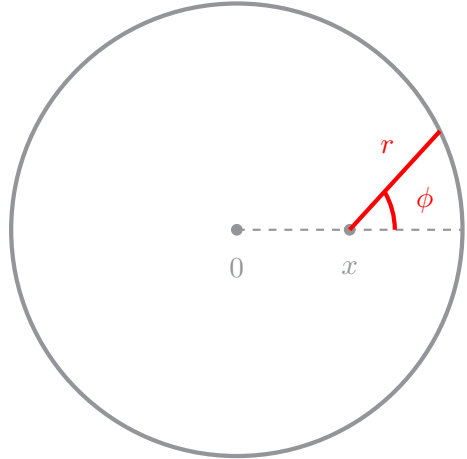

The variance can be computed for other disks by applying Theorem~\ref{thm:props}. More precisely, if $U = \{z : |z-w| = s \}$, then
\[ v_U(z) = s^{2} - \frac{ 4s^{2} E(|z-w|)^{2}}{\pi^{2}} . \]
As $E$ is a decreasing concave function on $(0,1)$, it follows that $v_{\mathbb{D}}$ is a convex function on $\mathbb{D}$.

Recall that to prove Theorem \ref{thm:dim2convex}, we need to show that $v_U$ is a convex function.
To prove this result, we need some preliminary lemmas. Let $z_{0} \in U$, $e^{i\sigma} \in S^1$ and $w\in \partial U$. Suppose that $y$ is in the direction $\phi$ relative to $z_{0}$, that is $y = z_{0} + se^{i\phi}$ for some $s>0$. Choose $\delta_{0}>0$ small enough that $z_{\delta} = z_{0} + \delta e^{i\sigma} \in U$ for $0\leq \delta \leq \delta_{0}$. If we fix $\delta \in (0,\delta_{0})$, let $\alpha$ be the direction of $y$ relative to $z_{\delta}$. While $\alpha$ does depend on $\delta$, we will typically suppress this in the notation. Moreover, if $\delta$ is small, then $\phi$ and $\alpha$ are close for any $y\in \partial U$.

Let $T$ be the triangle whose vertices are $z_{0}$, $z_{\delta}$, and $y$. Then the side lengths of $T$ are $r = d_{U}(z_{0},\phi), s = d_{U}(z_{\delta},\alpha)$, and $\delta$, and the internal angles opposite each respective side are $\pi - \alpha+\sigma$, $\phi - \sigma$, and $\alpha - \phi$ (see Figure~\ref{fig:3}).

\begin{figure}[ht]
\begin{center}
\begin{tikzpicture}
\begin{scope}[rotate=-45, scale=1.6]
\draw [gray, ultra thick] plot [smooth cycle] coordinates {(-1,-2.2) (1,-1.7) (1.7,-0.3) (1,1.8) (0,2.3) (-1,1.4) (-1.8, 0.4) (-2,-1.2)};
\end{scope}
\node[text=gray] at (-2.5,-1) {$U$};
\node[text=gray] at (0,-0.9) {$z_{0}$};
\node[text=red] at (1.5,0.1) {$z_{\delta}$};
\filldraw [gray] (0,-0.5) circle (2pt);
\draw[dashed, gray, ultra thick] (0,-0.5) -- (2.6,-0.5); \draw[dashed, gray, ultra thick] (0,-0.5) -- (3.1,1.5); 
\draw[gray, ultra thick] (0.9,-0.5) arc (0:30:1);
\node[text=gray] at (1.1,-0.2) {$\sigma$};
\draw[dashed, red, ultra thick] (0,-0.5) -- (0,2.7); 
\draw[red, >=triangle 45, <->] (-0.2,-0.4) -- (-0.2,2.6);
\node[text=red] at (-0.5,1.2) {$r$};
\draw[red, ultra thick] (0.5,-0.2) arc (30:90:0.6);
\node[text=red] at (0.6,0.4) {$\phi - \sigma$};
\filldraw [red] (0,2.7) circle (2pt);
\node[text=red] at (0,3.1) {$y$};
\draw[dashed, red, ultra thick] (0,2.7) -- (1.3,0.4); 
\draw[red, ultra thick] (1.7,0.6) arc (30:100:0.6);
\node[text=red] at (1.9,1.2) {$\alpha - \sigma$};
\draw[red, >=triangle 45, <->] (0.3,2.6) -- (1.2,1.1);
\node[text=red] at (1.3,1.7) {$s$};
\filldraw [red] (1.3,0.4) circle (2pt);
\end{tikzpicture}
\caption{The triangle $T$ with angles and side lengths.}
\label{fig:3}
\end{center}
\end{figure}
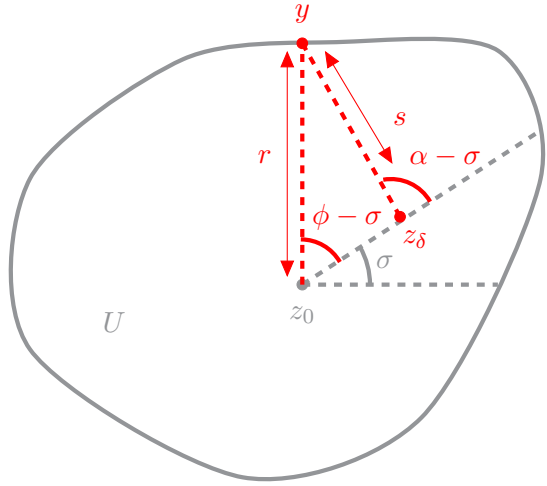

\begin{lemma}
\label{lem:d2var1}
With the notation as above,
\[ d_{U}(z_{\delta},\alpha) = d_{U}(z_{0},\phi) - \delta \cos (\phi - \sigma) + O(\delta^{2})\]
as $\delta \to 0$.
\end{lemma}

\begin{proof}
By the law of cosines applied to $T$, we have
\[ s^{2} = r^{2} +\delta^{2} - 2\delta r \cos ( \phi - \sigma) .\]
It follows that
\begin{align*}
s &= r \left( 1 - \frac{ 2\delta \cos (\phi  -\sigma) }{r} + O(\delta^{2}) \right)^{1/2}\\
&=r \left( 1 - \frac{\delta \cos (\phi - \sigma) }{r} + O(\delta^{2}) \right) \\
&=r - \delta \cos (\phi - \sigma) + O(\delta^{2}),
\end{align*}
as required.
\end{proof}

\begin{lemma}
\label{lem:d2var2}
With the notation as above,
\[ \sin (\alpha  -\sigma) = \sin (\phi - \sigma) + \delta \left( \frac{ \sin (\phi - \sigma) \cos (\phi - \sigma) }{d_{U}(z_{0},\phi) } \right) + O(\delta^{2}) \]
as $\delta \to 0$.
\end{lemma}

\begin{proof}
By the law of sines applied to $T$, the fact that $\sin(\pi - x) = \sin x$, and Lemma~\ref{lem:d2var1}, we have
\[ \frac{\sin (\alpha - \sigma)}{\sin (\phi - \sigma) } = \frac{\sin (\pi - \alpha + \sigma)}{\sin (\phi - \sigma)} = \frac{ r}{s} = \frac{ r }{r - \delta \cos (\phi - \sigma) + O(\delta^{2}) } .\]
It follows that
\begin{align*}
\sin (\alpha - \sigma) &= \sin (\phi - \sigma) \left( 1 - \frac{\delta \cos (\phi  - \sigma) }{r} + O(\delta^{2}) \right)^{-1} \\
&= \sin (\phi - \sigma) \left( 1 + \frac{ \delta \cos (\phi - \sigma) }{r} + O(\delta^{2}) \right),
\end{align*}
as required.
\end{proof}

Viewing $\alpha$ as a function of $\phi$, we may estimate the derivative as follows.

\begin{lemma}
\label{lem:d2var3}
With the notation as above,
\[ \frac{ d \alpha}{d \phi} = 1 + \delta \left( \frac { (\cos (\phi - \sigma)) d_{U}(z_{0},\phi)  - (\sin (\phi -\sigma)) d_{U}'(z_{0},\phi) } { d_{U}(z_{0},\phi ) ^{2}} \right) + O(\delta ^{2}) \]
as $\delta  \to 0$.
\end{lemma}

\begin{proof}
By the law of sines applied to $T$, we have
\[ \sin ( \alpha - \phi ) = \frac{ \delta \sin (\phi - \sigma) }{s} .\]
By Lemma~\ref{lem:d2var1}, it follows that
\begin{align*}
\alpha &= \phi + \arcsin  \left [ \frac{ \delta \sin (\phi - \sigma) }{s} \right ]\\
&= \phi + \arcsin \left [ \frac{ \delta \sin (\phi - \sigma) }{ r  - \delta \cos (\phi - \sigma) + O(\delta^{2}) } \right ]\\
&= \phi + \arcsin \left [ \frac{ \delta \sin (\phi - \sigma) }{r} \left( 1 - \frac{\delta \cos (\phi - \sigma)}{r} + O(\delta^{2}) \right)^{-1} \right ]\\
&= \phi + \arcsin \left [ \frac{ \delta \sin (\phi - \sigma) }{r} \left( 1 + \frac{\delta \cos (\phi - \sigma)}{r} + O(\delta^{2}) \right) \right ]\\
&= \phi +\frac{\delta \sin ( \phi - \sigma) }{r} + O(\delta^{2}).
\end{align*}
Recalling that $r = d_{U}(z_{0},\phi)$, we see by differentiating that
\[ \frac{ d \alpha}{d \phi} = 1+ \delta \left( \frac { (\cos (\phi - \sigma)) d_{U}(z_{0},\phi) - \sin  ( \phi - \sigma) d_{U} ' (z_{0},\phi) } {d_{U}(z_{0},\phi) ^{2}}  \right) + O(\delta^{2}), \] 
as required.
\end{proof}

With these lemmas in hand, we may now show that $v_U$ is convex.

\begin{proof}[Proof of Theorem~\ref{thm:dim2convex}]
By Theorem~\ref{thm:i2}, $I_{2}^{U}$ is constant. As $v_U(z) = I_{2}^{U}(z) - [I_{1}^{U}(z)]^{2}$, it follows that $v_U$ is strictly convex if and only if $I_{1}^{U}$ is strictly concave. It is enough to show that the directional second derivative of $I_{1}^{U}$ in every direction is strictly negative. This is what we will do.

Fix $z_{0} \in U$ and a direction $e^{i\sigma} \in S^1$. Choose $\delta >0$ small enough that $z_{\delta} = z_{0} + \delta e^{i\sigma } \in U$. For $y\in \partial U$, we set $\phi$ and $\alpha$ to be the directions of $y$ from $z_{0}$ and $z_{\delta}$ respectively.
Then by Lemma~\ref{lem:d2var1} and Lemma~\ref{lem:d2var3}, we have
\begin{align*}
I_{1}^{U}(z_{\delta}) &= \frac{1}{2\pi} \int_{0}^{2\pi} d_{U} (z_{\delta}, \alpha) \: d\alpha \\
&= \frac{1}{2\pi} \int_{0}^{2\pi} \left( d_{U}(z_{0},\phi) - \delta \cos (\phi - \sigma) + O(\delta^{2}) \right) \cdot \\
& \hskip0.5in \left( 1 + \delta \left( \frac { \cos (\phi - \sigma) d_{U}(z_{0},\phi)  - \sin (\phi  - \sigma)d_{U}'(z_{0},\phi) } { d_{U}(z_{0},\phi ) ^{2}} \right) + O(\delta ^{2}) \right)  \: d\phi\\
&= \frac{1}{2\pi} \int_{0}^{2\pi} \left( d_{U}(z_{0},\phi) - \frac{ \delta \sin (\phi - \phi)  d_{U}'(z_{0},\phi) }{d_{U}(z_{0},\phi)} + O(\delta^{2}) \right)  \: d\phi\\
&= I_{1}^{U}(z_{0}) - \frac{\delta}{2\pi} \int_{0}^{2\pi} \frac{ \sin (\phi - \sigma) d_{U}'(z_{0},\phi)}{d_{U}(z_{0},\phi)}  \: d\phi + O(\delta^{2}).
\end{align*}
It follows that the directional derivative of $I_{1}^{U}$ at $z_{0}$ in the direction $\sigma$ is
\begin{equation}
\label{eq:d2vareq1}
D_{\sigma} I_{1}^{U}(z_{0}) = - \frac{1}{2\pi}  \int_{0}^{2\pi} \frac{ \sin (\phi - \sigma) d_{U}'(z_{0},\phi) }{d_{U}(z_{0},\phi)} \: d\phi.
\end{equation}
For the second directional derivative, by \eqref{eq:d2vareq1}, we have
\[
D_{\sigma} I_{1}^{U}(z_{\delta}) =  - \frac{1}{2\pi}  \int_{0}^{2\pi} \frac{ \sin (\alpha - \sigma) d_{U}'(z_{\delta},\alpha) }{d_U(z_{\delta},\alpha)} \: d\alpha.
\]
By Lemma~\ref{lem:d2var1} and Lemma~\ref{lem:d2var3} applied to the inverse, we have
\begin{align*}
d_{U}'(z_{\delta},\alpha) &= \frac{ d\phi}{d\alpha} \cdot \frac{d}{d \phi} \left( d_{U}(z_{0},\phi) - \delta \cos (\phi - \sigma) + O(\delta^{2}) \right) \\
&= \left(  1 - \delta \left( \frac { \cos (\phi - \sigma) d_{U}(z_{0},\phi)  - \sin (\phi  - \sigma) d_{U}'(z_{0},\phi) } { d_{U}(z_{0},\phi ) ^{2}}\right)  + O(\delta ^{2}) \right)\cdot \\
&\hskip0.5in  \left( d_{U}'(z_{0},\phi) + \delta \sin (\phi - \sigma) + O(\delta^{2}) \right)\\
&= d_{U} ' (z_{0},\phi) + \delta \left [ \left( 1 + \left( \frac{ d_{U} ' (z_{0},\phi) }{d_{U}(z_{0},\phi)}\right)^{2} \right) \sin( \phi - \sigma) - \frac{d_{U}'(z_{0},\phi)}{ d_{U}(z_{0},\phi) } \cos (\phi - \sigma) \right ]+ O(\delta^{2})\\
&=: d_{U}'(z_{0},\phi) + \delta X + O(\delta^{2}).
\end{align*}
Then by Lemma~\ref{lem:d2var1}, Lemma~\ref{lem:d2var2} and Lemma~\ref{lem:d2var3}, and some omitted calculations, we have
\begin{align*}
D_{\sigma}I_{1}^{U}(z_{\delta}) &= -\frac{1}{2\pi} \int_{0}^{2\pi} \left( \sin (\phi - \sigma) + \delta \left( \frac{ \sin (\phi - \sigma) \cos (\phi - \sigma) }{d_{U}(z_{0},\phi) } \right) + O(\delta^{2}) \right) \cdot \\
& \hskip0.5in  \left( d_{U}'(z_{0},\phi)  + \delta X + O(\delta^{2}) \right)\cdot \left( d_{U}(z_{0},\phi) - \delta \cos (\phi - \sigma) + O(\delta^{2}) \right)^{-1} \cdot \\
&\hskip0.5in \left( 1 + \delta \left( \frac { \cos (\phi - \sigma) d_{U}(z_{0},\phi)  -
\sin (\phi - \sigma)d_{U}'(z_{0},\phi)} { d_{U}(z_{0},\phi ) ^{2}} \right) + O(\delta ^{2}) \right) \: d\phi \\
&= \frac{1}{2\pi} \int_{0}^{2\pi}  \frac{ \sin (\phi - \sigma)  d_{U}'(z_{0},\phi)}{d_{U}(z_{0},\phi)} \:d\phi - \frac{\delta}{2\pi} \int_{0}^{2\pi} \frac{ \sin^{2} (\phi - \sigma) }{d_{U}(z_{0},\phi)} \: d\phi + O(\delta^{2})\\
&= D_{\sigma}I_{1}^{U}(z_{0})  - \frac{\delta}{2\pi} \int_{0}^{2\pi} \frac{ \sin^{2} (\phi - \sigma) }{d_{U}(z_{0},\phi)} \: d\phi + O(\delta^{2}).
\end{align*}
It follows that 
\[ D_{\sigma}^{2} I_{1}^{U}(z_{0}) = -\frac{1}{2\pi} \int_{0}^{2\pi} \frac{ \sin^{2} (\phi - \sigma) }{d_{U}(z_{0},\phi)} \: d\phi .\]
As the integrand here is non-negative and not identically zero, we conclude that $D_{\sigma}^{2} I_{1}^{U}(z_{0})$ is strictly negative, which completes the proof.
\end{proof}

As $v_U$ is a strictly convex function, it has a unique minimum on $U$.

\begin{definition}
\label{def:variocentre}
Let $U \subset \mathbb{R}^{2}$ be a bounded convex domain. The unique minimum of $v_U$ is called the variocentre of $U$, denoted by $z_U$.
\end{definition}

By \eqref{eq:d2vareq1}, the unique minimum $z_U$ must be a critical point of $D_{\sigma}I_{1}^{U}$ for every direction $\sigma$, and this can be the only such point. This yields the following corollary.

\begin{corollary}
\label{cor:vc}
Let $U \subset \mathbb{R}^{2}$ be a bounded convex domain. Then $z_{0}$ is the variocentre if and only if
\[    \int_{0}^{2\pi} \frac{ \sin (\phi - \sigma) d_{U}'(z_{0},\phi) }{d_{U}(z_{0},\phi)} \: d\phi = 0 \]
for every $\sigma \in [0,2\pi)$, or equivalently,
\[   \int_{0}^{2\pi} \cos(\phi -\sigma) \ln d_{U}(z_{0},\phi) \: d\phi = 0\]
for all $\sigma \in [0,2\pi)$.
\end{corollary}

\begin{proof}
The first conclusion follows from the observation above. For the second conclusion, integrating by parts yields
\begin{align*}
\int_{0}^{2\pi} \frac{ \sin (\phi - \sigma) d_{U}'(z_{0},\phi) }{d_{U}(z_{0},\phi)} \: d\phi &=
\left [ \sin (\phi - \sigma) \ln d_{U}(z_{0},\phi) \right ]_{\phi = 0}^{2\pi} - \int_{0}^{2\pi} \cos(\phi -\sigma) \ln d_{U}(z_{0},\phi) \: d\phi \\
&= - \int_{0}^{2\pi} \cos(\phi -\sigma) \ln d_{U}(z_{0},\phi) \: d\phi,
\end{align*}
which gives the result.
\end{proof}

By the formula for $v_U$, we can obtain the following representation for the directional derivative $D_{\sigma} v_U$.

\begin{theorem}
\label{thm:vrep}
Let $U$ be a bounded convex domain in $\mathbb{R}^{2}$ and $e^{i\sigma} \in S^1$. Then for $z_{0} \in U$, we have
\[ D_{\sigma} v_U(z_{0}) = \frac{ I_{1}^{U}(z_{0}) }{\pi} \int_{\sigma-\pi/2}^{\sigma + \pi/2} \cos ( \phi -\sigma ) \ln \left( \frac{ d_{U}(z_{0},\phi+\pi) }{d_{U}(z_{0},\phi) } \right) \: d\phi .\]
\end{theorem}

\begin{proof}
As $v(z_{0}) = I_{2}(z_{0}) - [ I_{1}(z_{0}) ]^{2}$ and $I_{2}(z_{0})$ is constant by Theorem~\ref{thm:i2}, it follows that 
\[ D_{\sigma} v(z_{0}) = -2I_{1}(z_{0}) D_{\sigma} I_{1} (z_{0}).\]
Using \eqref{eq:d2vareq1}, integrating by parts as in the proof of Corollary~\ref{cor:vc}, and using $2\pi$-periodicity and the identity $\cos(\phi + \pi) = -\cos (\pi)$, this gives
\begin{align*}
D_{\sigma} v(z_{0}) &=  \frac{I_{1}(z_{0})}{\pi} \int_{0}^{2\pi}  \frac{ \sin (\phi -\sigma)d_{U}'(z_{0},\phi)}{d_{U}(z_{0},\phi)}   \:d\phi \\
&= - \frac{I_{1}(z_{0})}{\pi} \int_{0}^{2\pi} \cos(\phi - \sigma) \ln d_{U}(z_{0},\phi) \: d\phi \\
&=- \frac{I_{1}(z_{0})}{\pi} \int_{\sigma -\pi/2}^{\sigma + 3\pi /2} \cos(\phi - \sigma) \ln d_{U}(z_{0},\phi) \: d\phi \\
&=- \frac{I_{1}(z_{0})}{\pi} \int_{\sigma-\pi/2}^{\sigma + \pi/2} \cos (\phi - \sigma) \left( \ln d_{U} (z_{0},\phi) - \ln d_{U} (z_{0}, \phi +\pi) \right) \: d\phi,
\end{align*}
which gives the result.
\end{proof}

The point of the formulation in Theorem~\ref{thm:vrep} is that $\cos(\phi - \sigma)$ is non-negative for $\phi \in [\sigma -\pi/2 , \sigma + \pi/2]$ and so the parity of $D_{\sigma}v_U(z_{0})$ can, in principle, be determined by knowing how $d_{U}(z_{0},\phi)$ compares to $d_{U}(z_{0},\phi + \pi)$.
As a first illustration of this, if $U$ has rotational symmetry, then the variocentre is located at the centre of the rotation.

\begin{corollary}
\label{cor:rotation}
Let $U$ be a bounded convex domain in $\mathbb{R}^{2}$ and suppose that there exists $z_{0} \in U$ such that $d_{U} (z_{0}, \phi) = d_{U} (z_{0}, \phi + \pi)$ for all $\phi \in [0,2\pi)$. Then $z_{0}$ is the variocentre.
\end{corollary}

\begin{proof}
By Theorem~\ref{thm:vrep}, it follows that $D_{\sigma}v_U(z_{0}) = 0$ for all $\sigma \in S^1$. Thus $z_{0}$ is the unique critical point of the strictly convex function $v_U$ and is the variocentre.
\end{proof}

Next, we want to compare the change in the variance when compared to the distance to the boundary. First, it is not always true that the variance and the distance to the boundary are inversely related, as the following example shows.

\begin{example}
\label{ex:2d}
Let $U\subset \mathbb{R}^{2}$ be the convex quadrilateral with vertices at $(\pm1, 1)$ and $(\pm 2, -1)$, as shown in Figure~\ref{fig:3.5}. Let $z_{0} = 0$ and $\sigma = -\pi/2$. Then for $-\pi/4 \leq \phi-\sigma  \leq  \pi/4$ we have $d_U(\phi + \pi ) = d_U(\phi)$. However, for $-\pi/2 < \phi-\sigma < \pi/4$ or $\pi /4 < \phi -\sigma<  \pi/2$ we have $d_U(\phi+\pi) < d_U(\phi)$. By Theorem~\ref{thm:vrep}, we see that
\[ D_{-\pi/2} v_U(z_{0}) < 0.\]
However, heading in the direction with angle $-\pi/2$ from $z_{0} = 0$, we also see that the distance to the boundary is strictly decreasing. For this example, by symmetry, the variocentre lies at $(0,y)$ for some $y<0$.
\end{example}

\begin{figure}[ht]
\begin{center}
\includegraphics[height=0.4\linewidth]{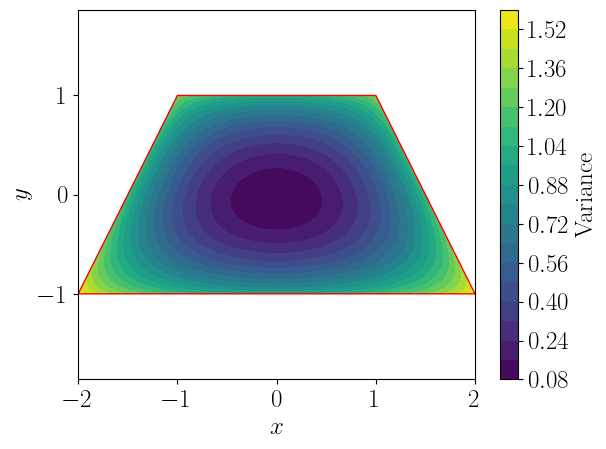}
\caption{Heatmap of the variance of the distance to the boundary, $v$, for the convex quadrilateral with vertices at $(\pm 1, 1)$ and $(\pm 2, -1)$.}
\label{fig:3.5}
\end{center}
\end{figure}

On the other hand, once we are close enough to the boundary of $U$, the variance and the distance to the boundary are inversely correlated, as the following result shows.

Note that in the statement of Theorem \ref{thm:vvskstatement}, $\sigma$ need not be unique: in fact, if $U$ is a disk and $z_{0}$ is its centre, then $e^{i\sigma}$ can take any value in $S^1$. The interpretation of this result is that near the boundary, the variance varies like $\ln d( \cdot, \partial U)^{-1}$. This is because $d(z_{\delta} , \partial U) = \delta r$ for $0<\delta <1$.

\begin{proof}[Proof of Theorem \ref{thm:vvskstatement}]
First, for $\delta \in (0,1)$, $w$ is the unique closest point to $z_{\delta}$ on $\partial U$. By Theorem~\ref{thm:vrep}, we need to estimate
\begin{equation}
\label{eq:vvskeq1} 
D_{\sigma} v_U(z_{\delta}) = \frac{I_{1}(z_{\delta})}{\pi} \int _{\sigma - \pi/2}^{\sigma +\pi /2} \cos(\phi - \sigma) 
\left( \ln \frac{1}{d_{U}(z_{\delta}, \phi) } + \ln d_{U}(z_{\delta},\phi+\pi) \right) \: d\phi.
\end{equation}
As $d_{U}(z_{\delta},\sigma) = \delta r$ and as $U$ is convex, the largest possible value that $d_{U}(z_{\delta}, \phi)$ can take, for $-\pi/2 < \phi -\sigma < \pi/2$, arises when the boundary of $U$ is a line meeting $\partial U$ perpendicular to the line segment joining $z_{0}$ to $w$, see Figure~\ref{fig:4}. By elementary trigonometry, this largest value is given by $\delta r / \cos(\phi - \sigma)$ for $-\pi/2 < \phi -\sigma< \pi /2$. 

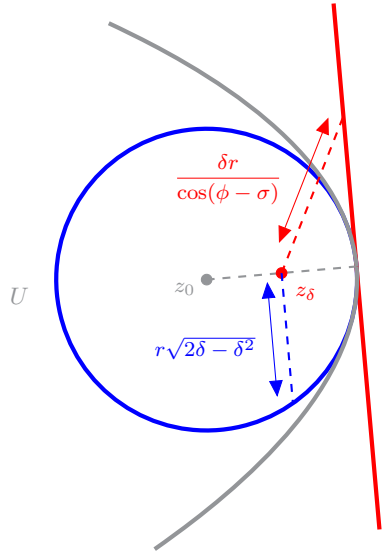
\begin{figure}[ht]
\begin{center}
\begin{tikzpicture}
\begin{scope}[rotate=5]
\node[text=gray] at (-2.5,0) {\footnotesize $U$};
\draw[blue, ultra thick](0,0) circle (2);
\draw[red, ultra thick] (2,-3.5) -- (2,3.5);
\draw[dashed, gray, thick] (0,0) -- (2,0);
\draw[gray, ultra thick, rotate=90, shift={(0,-2)}] (0,0) parabola (3.5,3);
\draw[gray, ultra thick, rotate=90, shift={(0,-2)}] (0,0) parabola (-3.5,3);
\filldraw[gray] (0,0) circle (2pt);
\node[text=gray] at (-0.3,-0.1) {
\scriptsize $z_{0}$};
\filldraw[red] (1,0) circle (2pt);
\node[text=red] at (1.3,-0.3) {
\scriptsize $z_{\delta}$};
\draw[dashed, blue, thick] (1,0) -- (1,-1.732);
\draw[blue, >=triangle 45, <->] (0.8,-0.1) -- (0.8,-1.632);
\node[text=blue] at (-0.1,-0.9) {
\scriptsize $r \sqrt{2 \delta - \delta^{2}}$};
\draw[dashed, red, thick] (1,0) -- (2,2);
\draw[red, >=triangle 45, <->] (1,0.4) -- (1.8,2);
\node[text=red] at (0.4,1.3) {
\scriptsize $\displaystyle\frac{\delta r}{\cos ( \phi - \sigma )}$};
\end{scope}
\end{tikzpicture}
\caption{Estimating distances to $\partial U$ from $z_{\delta}$.}
\label{fig:4}
\end{center}
\end{figure}

Since $U$ is bounded, this situation cannot occur for every $\phi - \sigma \in(-\pi/2,\pi/2)$, but this estimate is sufficient for our purposes. By changing variables with $t = \phi - \sigma$ and integrating by parts, we have
\begin{align*}
\int_{\sigma - \pi/2}^{\sigma + \pi/2} \cos ( \phi - \sigma) \ln \left( \frac{1}{d_{U}(z_{\delta}, \phi) } \right) \: d\phi 
&\geq \int_{\sigma - \pi/2}^{\sigma + \pi/2} \cos(\phi - \sigma) \ln \left( \frac{ \cos(\phi - \sigma) }{\delta r} \right) \: d\phi \\
&= \int_{-\pi/2}^{\pi/2} \left [ \cos (t) \ln( \cos (t) ) - \ln (\delta r) \cos (t) \right ] \: dt \\
&= \ln 4 -2 +2 \ln \frac{1}{\delta r}.
\end{align*}
Next, for $-\pi/2 < \phi -\sigma< \pi/2$, we note that $d_{U}(z_{\delta}, \phi + \pi)$ is at least the distance from $z_{\delta}$ to the boundary of the disk centred at $z_{0}$ of radius $r$ in the direction $\sigma + \pi/2$ (see again Figure~\ref{fig:4}). By elementary trigonometry, this gives the bound
\[ d_{U}(z_{\delta}, \phi + \pi) \geq \sqrt{ r^{2} - (1-\delta)^{2}r^{2}} = r\sqrt{2\delta - \delta^{2}}.\]
Therefore, we have
\begin{align*}
\int_{\sigma - \pi/2}^{\sigma+\pi/2} \cos ( \phi - \sigma) \ln d_{U} (z_{\delta},\phi + \pi) \:d\phi
&\geq \int_{\sigma - \pi/2}^{\sigma + \pi/2} \cos(\phi - \sigma) \ln ( r\sqrt{2\delta - \delta^{2}} ) \: d\phi \\
&= 2\ln r + \ln \delta + \ln ( 2 - \delta ).
\end{align*}
Combining these estimates with \eqref{eq:vvskeq1}, we obtain
\begin{align*}
D_{\sigma} v_U( z_{\delta} ) &\geq \frac{I_{1}(z_{\delta}) }{\pi} \left( \ln 4 - 2 +2\ln \left( \frac{1}{\delta r} \right) + 2\ln r +  \ln \delta + \ln ( 2 - \delta ) \right)\\
&= \frac{I_{1}(z_{\delta}) }{\pi} \left( \ln 4 -2 + \ln (2-\delta)+  \ln \frac{1}{\delta} \right), 
\end{align*}
from which the result follows.
\end{proof}

Theorem~\ref{thm:vvskstatement} shows that the variance increases as we move towards $\partial U$, once we are sufficiently close to $\partial U$. On the other hand, if $U$ is close to circular, then the variance increases as we move towards $\partial U$, once we are outside a compact set.

To put the final conclusion of Theorem~\ref{thm:ellipticstatement} into context, recall Example~\ref{ex:disk}, which shows that if $U$ is the unit disk then $D_{\arg z} v_U (z) > 0$ for every $z\in U \setminus \{0\}$. Of course, Theorem~\ref{thm:ellipticstatement} can be generalized to domains whose boundaries are close to circles other than the unit circle by an application of Theorem~\ref{thm:props}.

\begin{proof}[Proof of Theorem~\ref{thm:ellipticstatement}]
Without loss of generality, we may assume that $z = r$ lies on the positive real axis.
As $\partial U$ is contained in $\{z : 1 \leq |z| \leq 1+\epsilon \}$, we will estimate $D_{0} v_U(r)$ by replacing $U$ with the domain $V$ which is starlike about $r$ and whose boundary contains $\{ z : |z| =1, \operatorname{Re}(z) <r \}$ and $\{ z : |z| = 1+\epsilon, \operatorname{Re}(z) > r \}$ (see Figure~\ref{fig:5}). Note that $V$ is not convex for $\epsilon >0$.

\begin{figure}[ht]
\begin{center}
\begin{tikzpicture}
\draw[gray, ultra thick](0,0) circle (2);
\draw[gray, ultra thick](0,0) circle (3);
\draw[dashed, gray, thick] (0,0) -- (2,0);
\node[text=blue] at (0.6,0.9) {$V$};
\node[text=red] at (3.5,-0.4) {$1 + \epsilon$};
\filldraw[gray] (0,0) circle (2pt);
\node[text=gray] at (0,-0.4) {$0$};
\filldraw[gray] (1.1,0) circle (2pt);
\node[text=gray] at (1.1,-0.4) {$r$};
\filldraw[red] (2,0) circle (2pt);
\node[text=red] at (2.2,-0.4) {$1$};
\filldraw[red] (3,0) circle (2pt);
\node[text=red] at (3.5,-0.4) {$1 + \epsilon$};
\draw[dashed, blue, ultra thick] (1.1,1.832) arc (60:300:2.141);
\draw[dashed, blue, ultra thick] (1.1,1.832) -- (1.1,2.6);
\draw[dashed, blue, ultra thick] (2.9, 0) arc (0:68:2.9);
\draw[dashed, blue, ultra thick] (1.1,-1.832) -- (1.1,-2.6);
\draw[dashed, blue, ultra thick] (2.9, 0) arc (0:-68:2.9);
\end{tikzpicture}
\caption{The domain $V$.}
\label{fig:5}
\end{center}
\end{figure}
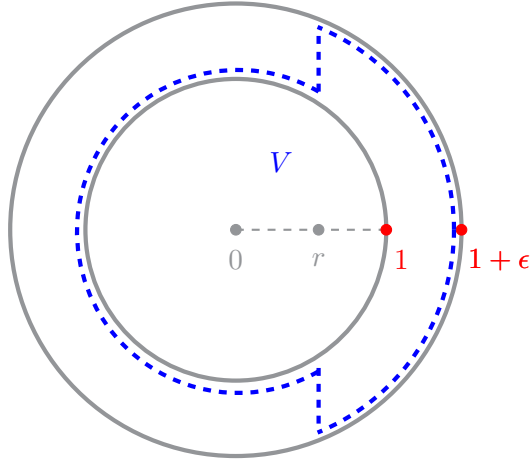

For $-\pi /2 < \phi < \pi /2$ we have
\[ d_U(r, \phi) \leq d_V(r,\phi),\]
whereas for $\pi/2 < \phi < 3\pi /2$ we have
\[ d_U(r,\phi) \geq d_V(r,\phi).\]
We can explicitly compute $d := d_V(r,\phi)$ as follows. First, for $\pi < \phi < 3\pi/2$ we have
\[ | r + de^{i\phi} | = 1.\]
This leads to the quadratic equation
\[ d^{2} + 2dr \cos \phi + r^{2}-1 = 0.\]
As $d>0$, taking the appropriate root of this equation yields
\[ d_V(r,\phi) = -r\cos \phi  + (1-r^{2}\sin^{2} \phi )^{1/2} \]
for $\pi/2 < \phi < 3\pi/2$. Next, for $-\pi /2 < \phi < \pi/2$ we have
\[ |r+de^{i\phi} | = 1+\epsilon .\]
Again solving the associated quadratic equation, this yields
\[ d_V(r,\phi) = -r\cos \phi + ( (1+\epsilon)^{2} - r^{2}\sin^{2} \phi )^{1/2}\]
for $-\pi/2 < \phi < \pi/2$. Then, by Theorem~\ref{thm:vrep}, we have
\begin{align*} 
D_{0} v_U(r) &= \frac{ I_{1}^{U}(r) }{\pi } \int_{-\pi/2}^{\pi/2} \cos \phi \ln \left( \frac{d_U(r , \phi + \pi)}{ d_U(r , \phi )} \right) \: d\phi \\
&\geq \frac{ I_{1}^{U}(r) }{\pi} \int_{-\pi/2}^{\pi/2} \cos \phi   \ln \left( \frac{d_V(r , \phi + \pi)}{ d_V(r , \phi )} \right) \: d\phi \\
&= \frac{ I_{1}^{U}(r) }{\pi} \int_{-\pi/2}^{\pi/2} \cos \phi \ln \left( \frac{ -r\cos (\phi + \pi) + \sqrt{1-r^{2}\sin^{2}(\phi + \pi ) } }{-r\cos \phi + \sqrt{ (1+\epsilon)^{2} - r^{2}\sin^{2}\phi} } \right) \: d\phi \\
&= \frac{I_{1}^{U}(r)}{\pi} \int_{-\pi/2} ^{\pi/2} \cos \phi \ln \left( \frac{ r\cos \phi  + \sqrt{ 1 - r^{2}\sin^{2} \phi}  }
{-r\cos \phi  + \sqrt{(1+\epsilon)^{2}-r^{2}\sin^{2} \phi }} \right)\: d\phi .
\end{align*}
To deal with this integral, we first compute the following integral:
\[ \int_{-\pi/2}^{\pi/2} \cos \phi \ln \left( r\cos \phi  + \sqrt{1-r^{2}\sin^{2} \phi} \right) \: d\phi = 
2 \int_{0}^{\pi/2} \cos \phi \ln \left( r\cos \phi  +\sqrt{1-r^{2}\sin^{2} \phi} \right) \: d\phi.\]
Making the substitution $u = r\sin\phi$ and then integrating by parts, we obtain
\begin{align*}
2 \int_{0}^{\pi/2} \cos \phi &\ln \left( r\cos \phi  + \sqrt{1-r^{2}\sin^{2} \phi} \right) \: d\phi 
= \frac{2}{r} \int_{0}^r \ln \left( \sqrt{r^{2}-u^{2}} + \sqrt{1-u^{2}} \right) \: du \\
&= \left [ \frac{2u}{r}  \ln \left( \sqrt{r^{2}-u^{2}} + \sqrt{1-u^{2}} \right) \right ]_{u=0}^r + \frac{2}{r} \int_{0}^r \frac{ u^{2}}{\sqrt{1-u^{2}}\sqrt{r^{2}-u^{2}} } \: du \\
&= \ln (1-r^{2}) + \frac{2}{r} \int_{0}^r \frac{ 1}{\sqrt{1-u^{2}}} \left( \frac{r^{2}}{\sqrt{r^{2}-u^{2}}} - \sqrt{r^{2}-u^{2}} \right) \: du\\
&=\ln(1-r^{2}) + 2rK(r) - \frac{2}{r} \int _{0}^r \left( \frac{\sqrt{1-u^{2}}}{\sqrt{r^{2}-u^{2}}} + \frac{r^{2}-1}{ \sqrt{1-u^{2}}\sqrt{r^{2}-u^{2}} } \right)\: du\\
&= \ln(1-r^{2}) +2rK(r) -\frac{2}{r} \left( E(r) + (r^{2}-1) K(r) \right) \\
&= \ln(1-r^{2}) + \frac{2}{r} \left( K(r) - E(r) \right).
\end{align*}
Here we have used \eqref{eq:ellK} and \eqref{eq:ellE} to write our integral using elliptic integrals. By an analogous computation that we omit, we find that
\begin{align*} 
\int_{-\pi/2}^{\pi/2} \cos \phi &\ln \left( -r\cos \phi + ((1+\epsilon)^{2} - r^{2}\sin^{2}\phi )^{1/2} \right) \: d\phi \\
&=  \ln ( (1+\epsilon)^{2}- r^{2} ) - \frac{ 2(1+\epsilon)}{r} \left( K(r/(1+\epsilon) ) - E(r/(1+\epsilon)) \right).
\end{align*}
Putting this altogether, we obtain
\begin{align}
\nonumber D_{0} v_U(r) & \geq \frac{ I_{1}^{U}(r)}{\pi} \left [ \ln \left( \frac{1-r^{2}}{(1+\epsilon)^{2}-r^{2}} \right) +  \frac{2}{r} \left( K(r) - E(r) \right) \right. \\
\label{eq:fre} & \hskip0.75in + \left. \frac{ 2(1+\epsilon)}{r} \left( K(r/(1+\epsilon) ) - E(r/(1+\epsilon)) \right) \right ].
\end{align}

Let us define $F(r,\epsilon)$ to be the function given by the expression in the square brackets in \eqref{eq:fre}. Then for $\epsilon >0$ and $r\in (0,1)$, $F$ is evidently a continuous function. Moreover, by \eqref{eq:ellKE}, we have
\begin{equation}
\label{eq:fre1} 
\lim_{\epsilon \to 0} F(r,\epsilon) = \frac{4}{r} \left( K(r) - E(r) \right)  = 4\pi \sum_{n=1}^{\infty} \left( \frac{ (2n-1)!!}{(2n)!!} \right)^{2} r^{2n-1}.
\end{equation}
Thus $F(r,\epsilon)$ extends continuously to $\epsilon = 0$ and we have $F(r,0) > 0$ for $r>0$. 

Now suppose $\epsilon >0$ is fixed. As $r\to 0^{+}$, we see that
\begin{equation}
\label{eq:f0e} 
F(r, \epsilon) \to -2 \ln (1+\epsilon) < 0.
\end{equation}
As $r\to 1^-$, we may use the asymptotic expansions \eqref{eq:K1} and \eqref{eq:E1} to give
\begin{align*}
F(r,\epsilon) &= \ln (1-r) + \ln (1+r) - \ln ( (1+\epsilon)^{2}- r^{2}) \\
& \hskip0.5in + \frac{2}{r} \left(-\tfrac12 \ln(1-r) \left( 1+O(1-r) \right) + \ln 4 -1 +O(1-r) \right) \\
& \hskip0.5in + \frac{ 2(1+\epsilon)}{r} \left( K(r/(1+\epsilon) ) - E(r/(1+\epsilon)) \right) \\
&= \left( 1 - \frac{1}{r} \right)\ln (1-r) + \ln 16 -2 - \ln(2\epsilon +\epsilon^{2}) + O(1-r) \\
& \hskip0.5in + \frac{ 2(1+\epsilon)}{r} \left( K(r/(1+\epsilon) ) - E(r/(1+\epsilon)) \right)\\
&\to \ln 16 - 2  - \ln(2\epsilon +\epsilon^{2}) + 2(1+\epsilon) \left( K(1/(1+\epsilon) ) - E(1/(1+\epsilon)) \right).
\end{align*}
In particular, there exists $\epsilon_{0}>0$ such that
\[ \lim_{r\to 1^-} F(r,\epsilon) > 0 \]
for all $\epsilon \in (0,\epsilon_{0})$. Moreover, combining this with \eqref{eq:f0e} and using the continuity of $F$, we see that given $\epsilon \in (0,\epsilon_{0})$, there exists $r(\epsilon)\in(0,1)$ such that $F(r,\epsilon) > 0$ for $r>r(\epsilon)$.

For the final claim, suppose that $r>0$ is fixed. Then using \eqref{eq:fre1}, we see that for $\epsilon >0$ sufficiently small, we have $F(r,\epsilon )>0$. We conclude that $r(\epsilon) \to 0$ as $\epsilon \to 0$.
\end{proof}

\section{Dimension \texorpdfstring{$3$}{3}} \label{sec:Dimension_3}

Moving into dimension $3$, we will use spherical coordinates $(\phi, \theta) \in [0,\pi] \times [0,2\pi)$ for $S^{2}$ so that $(\phi, \theta) = (0,0)$ corresponds to the direction $(1,0,0)$ in Euclidean coordinates. 
One immediate complication is that $I_{2}^{U}$ is not constant in dimensions greater than 2. For our first example, we recall the case of the unit ball from \cite{Straw23}. Unfortunately, the computation there was incorrect. Here we provide the correction.

\begin{example}
\label{ex:ball}
Let $U = \mathbb{B}^{3}$ be the unit ball in $^{3}$. It is clear that $I_{1}^{\mathbb{B}^{3}}(0) = 1 = I_{2}^{\mathbb{B}^{3}}(0)$ and so $v_{\mathbb{B}^{3}}(0) = 0$.
By rotational symmetry, $v_{\mathbb{B}^{3}}(x) = v_{\mathbb{B}^{3}} ( |x|)$ and so we may assume that $x=(r,0,0)$ for $0<r<1$.

Analogously to Example~\ref{ex:disk}, if $\sigma = (\phi, \theta) \in S^{2}$, then
\[ d_{\mathbb{B}^{3}}(x,\sigma) = -r\cos \phi + \sqrt{1-r^{2} \sin^{2} \phi } = -r\cos \phi + \sqrt{1-r^{2} + r^{2}\cos^{2}\phi } ,\]
noting in particular that this doesn't depend on $\theta$. Then, making the substitutions $u=r\cos \phi$ and $u = \sqrt{1-r^{2}}\tan t$, we have
\begin{align*}
I_{1}^{\mathbb{B}^{3}}(x) &= \frac{1}{4\pi} \int_{0}^{2\pi} \int_{0}^{\pi} d_{\mathbb{B}^{3}}(x,\sigma) \sin \phi \: d\phi \: d\theta \\
&= \frac{1}{2} \int_{0}^{\pi} \left( -r \cos \phi + \sqrt{ 1-r^{2} + r^{2} \cos ^{2}\phi } \right) \sin \phi \: d\phi \\
&= \frac{1}{2r} \int_{-r}^r \left( -u + \sqrt{1-r^{2} + u^{2}} \right) \: du \\
&= \frac{1}{r} \int_{0}^r \sqrt{1-r^{2} + u^{2} } \: du \\
&=\frac{1}{r} \int_{0}^{ \tan^{-1} (r/\sqrt{1-r^{2}})} \sqrt{ 1-r^{2} + (1-r^{2})\tan^{2}t} \cdot  \sqrt{1-r^{2}} \sec^{2}t \: dt \\
&= \frac{1-r^{2}}{r} \int_{0}^{\tan^{-1}(r/\sqrt{1-r^{2}} )} \sec^{3} t \: dt \\
&= \frac{1-r^{2}}{r} \left [ \tfrac12 \sec t \tan t  +  \tfrac12 \ln | \sec t  + \tan t | \right] _{t=0}^{ \tan^{-1}( r/\sqrt{1-r^{2}})} \\
&= \frac{1}{2} + \left( \frac{ 1-r^{2}}{4r} \right) \ln \left( \frac{1+r}{1-r} \right).
\end{align*}
Next, we have
\begin{align*}
I_{2}^{\mathbb{B}^{3}}(x) &= \frac{1}{4\pi} \int_{0}^{2\pi} \int_{0}^{\pi}  \left( -r \cos \phi + \sqrt{ 1-r^{2} + r^{2} \cos ^{2}\phi } \right)^{2} \sin \phi \: d\phi \:d\theta \\
&= \frac{1}{2} \int_{0}^{\pi} \left( 1-r^{2} + 2r^{2} \cos^{2}\phi - 2r\cos \phi \sqrt{ 1-r^{2} + r^{2}\cos^{2}\phi } \right) \sin \phi \: d\phi \\
&= \frac{1-r^{2}}{2} \int_{0}^{\pi} \sin \phi \: d\phi + r^{2} \int_{0}^{\pi} \cos^{2} \phi \sin \phi \: d\phi - r \int_{0}^{\pi} \cos \phi \sin \phi \sqrt{1-r^{2} + r^{2} \cos^{2}\phi} \: d\phi \\
&= (1-r^{2}) + \frac{2r^{2}}{3} + 0\\
&= 1-\frac{r^{2}}{3},
\end{align*}
where the third integral of the three above is zero as the integrand is antisymmetric about $\phi \to \pi - \phi$. We conclude that
\[ v_{\mathbb{B}^{3}}(x) = 1- \frac{r^{2}}{3} - \left [ \frac{1}{2} +  \left( \frac{ 1-r^{2}}{4r} \right) \ln \left( \frac{1+r}{1-r} \right) \right ]^{2}.\]
In particular, this confirms that $v_{\mathbb{B}^{3}}(0) = 0$ and we see that
\[ \lim_{r\to 1} v_{\mathbb{B}^{3}}( r,0,0) = \frac{5}{12}.\]
By computing the derivatives of the function above, which we leave as an exercise for the reader, it is evident that $v_{\mathbb{B}^{3}}(r,0,0)$ is a non-negative, increasing, convex function of $r$, and so $v_{\mathbb{B}^{3}}$ is a convex function on $\mathbb{B}^{3}$.
\end{example}

Turning now to regularity properties of the variance, we will compute its directional derivative. Given a bounded convex domain $U \subset ^{3}$, $x_{0} \in U$, and $\sigma \in S^{2}$, we may apply an isometry that moves $x_{0}$ to the origin and changes direction $\sigma$ to the direction given by $(\phi, \theta) = (0,0)$. This will make our spherical integrals straightforward to set up.

\begin{theorem}
\label{thm:d3pds}
Let $U\subset \mathbb{R}^{3}$ be a bounded convex domain and suppose that $x_{0}=0 \in U$. Set $\sigma_{0} = (\phi, \theta) = (0,0) \in S^{2}$. Then
\[ D_{\sigma_{0}} v(0)  = \frac{1}{4\pi} \int_{0}^{2\pi} \int_{0}^{\pi} (\sin ^{2} \phi ) (d_U(0,\phi , \theta))_{\phi} \left( \frac{2I_{1}^{U}(0) }{d_U(0,\phi, \theta)} - 1\right) \: d\phi \: d\theta ,\]
where $(d_U(0,\phi , \theta))_{\phi}$ denotes the partial derivative of the distance to the boundary function with respect to $\phi$.
\end{theorem}

\begin{proof}
Let $\delta < d(x_{0} , \partial U)$. Set $x_{\delta} = x_{0} + \delta \sigma_{0} \in U$. Given $y\in \partial U$, we will write spherical coordinates $(\phi , \theta)$ to denote the direction from $x_{0}$ to $y$, and spherical coordinates $(\alpha , \theta)$ to denote the direction from $x_{\delta}$ to $y$. Note that the $\theta$ coordinate does not change, due to the way we have set our construction up.
Then
\begin{align*}
I_{1}^{U}(x_{\delta}) &= \frac{1}{4\pi} \int_{S^{2}} d_U(x_{\delta}, \sigma) \: d\sigma\\
&= \frac{1}{4\pi} \int_{0}^{2\pi} \int_{0}^{\pi} d_U(x_{\delta}, \alpha, \theta) \sin \alpha \: d\alpha \: d\theta.
\end{align*}
Let $T$ be the triangle with vertices $x_{0},x_{\delta}$, and a point $y\in \partial U$. Then we may use Lemma~\ref{lem:d2var1}, Lemma~\ref{lem:d2var2}, and Lemma~\ref{lem:d2var3} just as in the two-dimensional case to obtain
\begin{align*}
I_{1}^{U}(x_{\delta}) &= \frac{1}{4\pi} \int_{0}^{2\pi} \int _{0}^{\pi } \left( d_U(x_{0} , \phi, \theta) - \delta \cos \phi + O(\delta^{2}) \right)
\cdot \left( \sin \phi + \frac{ \delta \sin \phi \cos \phi }{d_U(x_{0} , \phi , \theta ) } + O(\delta^{2}) \right)\\
&\hskip1in \cdot \left(  1 + \delta \left( \frac { (\cos \phi) d_U(x_{0} , \phi , \theta )  - (\sin \phi ) (d_U(x_{0},\phi, \theta))_{\phi} } { d_U(x_{0},\phi , \theta) ^{2}} \right) + O(\delta ^{2}) \right) \: d \phi \: d\theta \\
&= I_{1}^{U}(x_{0}) + \frac{\delta}{4\pi } \int_{0}^{2\pi} \int_{0}^{\pi} \left( \sin \phi \cos \phi -  \frac{ (\sin^{2} \phi )(d_U(x_{0},\phi, \theta))_{\phi}}{d_U(x_{0},\phi,\theta)} \right) \: d\phi \: d\theta + O(\delta^{2})\\
&= I_{1}^{U}(x_{0}) - \frac{\delta}{4\pi} \int_{0}^{2\pi} \int _{0}^{\pi} \frac{ (\sin^{2} \phi ) (d_U(x_{0},\phi, \theta))_{\phi} }{d_U(x_{0},\phi,\theta)} \: d\phi \: d\theta + O(\delta^{2}).
\end{align*}

In a similar manner, we may use  Lemma~\ref{lem:d2var1}, Lemma~\ref{lem:d2var2}, and Lemma~\ref{lem:d2var3} to compute $I_{2}^{U}(x_{\delta})$, omitting some of the details:
\begin{align*}
I_{2}^{U}(x_{\delta}) &= \frac{1}{4\pi } \int_{0}^{2\pi} \int _{0}^{\pi } d_U(x_{\delta}, \alpha , \theta)^{2} \: \sin \alpha \: d\alpha \: d\theta \\
&= \frac{1}{4\pi} \int_{0}^{2\pi} \int _{0}^{\pi } \left( d_U(x_{0} , \phi, \theta) - \delta \cos \phi + O(\delta^{2}) \right)^{2}
\cdot \left( \sin \phi + \frac{ \delta \sin \phi \cos \phi }{d_U(x_{0} , \phi , \theta ) } + O(\delta^{2}) \right)\\
&\hskip1in \cdot \left(  1 + \delta \left( \frac { (\cos \phi) d_U(x_{0} , \phi , \theta )  - (\sin \phi ) (d_U(x_{0},\phi, \theta))_{\phi} } { d_U(x_{0},\phi , \theta) ^{2}} \right) + O(\delta ^{2}) \right) \: d \phi \: d\theta \\
&= I_{2}^{U}(x_{0}) - \frac{\delta}{4\pi} \int_{0}^{2\pi} \int_{0}^{\pi} (\sin ^{2} \phi )  (d_U(x_{0},\phi, \theta))_{\phi} \: d\phi \: d\theta + O(\delta^{2}).
\end{align*}
Using the shorthand $d_{U} \equiv d_{U}(x_{0}, \phi, \theta)$ and $d_{U \phi} \equiv \left(d_{u}(x_{0}, \phi, \theta)\right)_{\phi}$, the variance is then given by
\begin{align*}
v_U(x_{\delta}) &= I_{2}^{U}(x_{\delta}) - [I_{1}^{U}(x_{\delta})]^{2} \\
&= I_{2}^{U}(x_{0}) - \frac{\delta}{4\pi} \int_{0}^{2\pi} \int_{0}^{\pi} (\sin ^{2} \phi ) d_{U \phi} \: d\phi \: d\theta + O(\delta^{2}) \\
&\hskip0.5in - \left [ I_{1}^{U}(x_{0}) - \frac{\delta}{4\pi} \int_{0}^{2\pi} \int _{0}^{\pi} \frac{ (\sin^{2} \phi ) d_{U \phi}}{d_{U}} \: d\phi \: d\theta + O(\delta^{2}) \right ]^{2}\\
&= v_U(x_{0}) + \frac{\delta }{4\pi } \int_{0}^{2\pi }\int_{0}^{\pi} \left [ -(\sin ^{2} \phi ) d_{U \phi} + \frac{2 I_{1}^{U}(x_{0}) (\sin^{2} \phi ) d_{U \phi} }{d_{U}} \right ] \: d\phi \: d\theta + O(\delta^{2}) \\
&= v_U(x_{0}) + \frac{\delta }{4\pi } \int_{0}^{2\pi }\int_{0}^{\pi} (\sin^{2}\phi )(d_{U \phi} \left( \frac{ 2I_{1}^{U}(x_{0})}{d_{U}} -1 \right) \: d\phi \: d\theta + O(\delta^{2}).
\end{align*}
The result then follows, since
\[ D_{\sigma_{0}} v_U(x_{0}) = \lim_{ \delta \to 0} \frac{ v_U(x_{\delta}) - v_U(x_{0})}{\delta}.\]
\end{proof}

Using integration by parts, we may represent the directional derivative as follows.

\begin{theorem}
\label{thm:vrep3d}
Let $U\subset \mathbb{R}^{3}$ be a bounded convex domain and suppose that $x_{0}=0 \in U$. Set $\sigma_{0} = (\phi, \theta) = (0,0) \in S^{2}$. Then
\[ D_{\sigma_{0}} v(0) = \frac{1}{4\pi} \int_{0}^{2\pi} \int_{0}^{\pi/2} \sin (2\phi ) G(\phi , \theta) \: d\phi \: d\theta,\]
where
\begin{equation}
\label{eq:gpt} 
G(\phi , \theta) = 2I_{1}^{U}(0) \ln \left( \frac{ d_U(0,\phi-\pi , \theta)}{d_U(0,\phi , \theta) } \right) + \left( d_U(0,\phi,\theta) - d_U(0,\phi - \pi , \theta) \right) .\end{equation}
\end{theorem}

\begin{proof}
Integrating the formula for $D_{\sigma_{0}}v(0)$ in Theorem~\ref{thm:d3pds} by parts with respect to $\phi$ by taking $s = \sin^{2} \phi$ and 
\[ dt =  (d_U(0,\phi , \theta))_{\phi} \left( \frac{2I_{1}^{U}(0) }{d_U(0,\phi, \theta)} - 1\right) \: d\phi,\]
we obtain $ds = \sin 2\phi \: d\phi$ and
\[ t = 2I_{1}^{U}(0) \ln d_U(0,\phi , \theta) - d_U(0,\phi , \theta).\]
Therefore
\begin{align*}
 \int_{0}^{\pi} (\sin ^{2} \phi ) & (d_U(0,\phi , \theta))_{\phi} \left( \frac{2I_{1}^{U}(0) }{d_U(0,\phi, \theta)} - 1\right) \: d\phi \\
&= \left [ (\sin^{2} \phi) v \right ]_{\phi = 0}^{\pi}  - \int_{0}^{\pi}  \sin 2\phi \left( 2I_{1}^{U}(0) \ln d_U(0,\phi , \theta) - d_U(0,\phi , \theta) \right) \: d\phi \\
&=   - \int_{0}^{\pi}  \sin 2\phi \left( 2I_{1}^{U}(0) \ln d_U(0,\phi , \theta) - d_U(0,\phi , \theta) \right) \: d\phi.
\end{align*}
As $\sin 2(\pi - \phi) = -\sin 2\phi$, we have
\begin{align*}
 -\int_{0}^{\pi}  \sin 2\phi & \left( 2I_{1}^{U}(0) \ln d_U(0,\phi , \theta) - d_U(0,\phi , \theta) \right) \: d\phi \\
&= - \int_{0}^{\pi/2}  \sin 2\phi \left( 2I_{1}^{U}(0) \ln d_U(0,\phi , \theta) - d_U(0,\phi , \theta) \right) \: d\phi  \\
& \hskip1in- \int_{\pi/2}^{\pi}\sin 2\phi \left( 2I_{1}^{U}(0) \ln d_U(0,\phi , \theta) - d_U(0,\phi , \theta) \right) \: d\phi  \\
&= - \int_{0}^{\pi/2}  \sin 2\phi \left( 2I_{1}^{U}(0) \ln d_U(0,\phi , \theta) - d_U(0,\phi , \theta) \right) \: d\phi \\
& \hskip1in +\int_{0}^{\pi/2} \sin 2\phi  \left( 2I_{1}^{U}(0) \ln d_U(0,\phi -\pi, \theta) - d_U(0,\phi -\pi , \theta) \right) \: d\phi,
\end{align*}
from which the theorem follows.
\end{proof}

It follows from Theorem~\ref{thm:vrep3d} that $G(\phi , \theta) = 0$ for any $(\phi , \theta) \in S^{2}$ if $0$ is the centre of a ball. Moreover, if $G(\phi , \theta) > 0$ for all $0<\phi < \pi/2$ and all $\theta \in (0,2\pi)$, then $D_{\sigma_{0}} v(0) > 0$. Yet the two terms in $G$ are competing against each other, which makes verifying this more challenging than the two-dimensional setting. However, we immediately have the following corollary of Theorem~\ref{thm:vrep3d}.

\begin{corollary}
\label{cor:symmetry3d}
Suppose that $U\subset \mathbb{R}^{3}$ is a bounded convex domain, that $x_{0} = 0\in U$ and $\sigma _{0} = (\phi , \theta) = (0,0) \in S^{2}$. If $d_U(0,\phi,\theta) = d_U(0,\phi-\pi, \theta)$ for all $\phi \in (0,\pi/2)$, then $D_{\sigma_{0}} v(0) = 0$ and $0$ is a critical point of the variance function.
\end{corollary}

If we knew that $v_U$ is convex in three dimensions then Corollary~\ref{cor:symmetry3d} would show that the symmetry properties in the hypothesis would guarantee the location of the unique minimum of the variance function. Even without knowing the convexity of $v_U$, we may still analyze certain properties of the variance. 

\begin{proof}[Proof of Theorem~\ref{thm:vvsk3dstatement}]
By applying a preliminary ambient isometry of $\mathbb{R}^{3}$, we may assume that $x_{0} = 0$ and $\sigma = (\phi , \theta) = (0,0)$. Then by 
Theorem~\ref{thm:vrep3d}, we have
\[ D_{\sigma_{0}} v(x_{\delta}) = \frac{1}{4\pi} \int_{0}^{2\pi} \int_{0}^{\pi/2} \sin (2\phi ) G(\phi , \theta) \: d\phi \: d\theta,\]
recalling \eqref{eq:gpt}.
First, it is clear that
\[ \left | d_U(x_{\delta},\phi,\theta) - d_U(x_{\delta},\phi - \pi , \theta) \right | \leq \operatorname{diam} (U).\]
To estimate the other term in $G(\phi,\theta)$, we will follow the strategy of Theorem~\ref{thm:vvskstatement} (recall again Figure~\ref{fig:4}). More precisely, given $x_{0}$ as in the hypotheses and $\sigma$ a direction to a nearest point on $\partial U$, it follows that this point is the unique nearest point to $x_{\delta}$, for $0<\delta <1$, on $\partial U$. Using trigonometry, it follows that for $0\leq \theta \leq 2\pi$ and $0\leq \phi < \pi/2$ we have
\[ d_U(x_{\delta} , \phi , \theta) \leq \frac{\delta r}{\cos \phi } .\]
On the other hand, for $0\leq \theta \leq 2\pi$ and $\pi/2 < \phi < \pi$, we have
\[ d_U(x_{\delta}, \phi , \theta) \geq r \sqrt{2\delta - \delta^{2}}.\]
We therefore have
\[ \frac{ d_U(x_{\delta},\phi-\pi , \theta)}{d_U(x_{\delta},\phi , \theta) } \geq \frac{  r \sqrt{2\delta - \delta^{2}}\cos \phi }{\delta r} = \left( \frac{2}{\delta} -1 \right)^{1/2} \cos \phi .\]
Putting this altogether, we may estimate $G(\phi , \theta)$ by
\[ G(\phi , \theta) \geq 2I_{1}^{U}(x_{\delta}) \left( \frac{1}{2} \ln \left( \frac{2}{\delta} - 1 \right) + \ln \cos \phi \right) - \operatorname{diam} U.\]
Thus
\[ D_{\sigma_{0}}v(x_{\delta}) \geq \frac{1}{4\pi} \int_{0}^{2\pi} \int_{0}^{\pi/2} \sin (2\phi) \left [  2I_{1}^{U}(x_{\delta}) \left( \frac{1}{2} \ln \left( \frac{2}{\delta} - 1 \right) + \ln \cos \phi \right) - \operatorname{diam} U \right ] \: d\phi \: d\theta.\]
The only term we need to check here involves the following integrals, where we make the substitution $u=\cos \phi$:
\begin{align*}
\int_{0}^{\pi/2} \sin (2\phi) \ln \cos \phi \: d\phi &= \int_{0}^{\pi/2} 2\sin \phi \cos \phi \ln \cos \phi \: d\phi \\
&= \int_{1}^0 - 2u\ln u \: du \\
&=2 \int_{0}^{1} u \ln u \: du = -\frac{1}{2}.
\end{align*}
It follows that 
\[ D_{\sigma_{0}} v(x_{\delta}) \geq \frac{1}{4\pi} \int_{0}^{2\pi} \int_{0}^{\pi/2} I_{1}^{U}(x_{\delta}) \sin(2\phi) \ln \left( \frac{2}{\delta} -1 \right) \: d\phi \: d\theta + O(1) = O\left(\ln  \frac{1}{\delta} \right) \]
as $\delta \to 0$. This completes the proof.
\end{proof}

Recall from Example~\ref{ex:ball} that if $U$ is the unit ball in $\mathbb{R}^{3}$, then $D_{x/|x|} v_U(x) >0$ for all $x\in U \setminus  \{ 0 \}$. If we perturb $U$ to be close to a ball, then once we are outside a compact subset of $U$, we have a similar result for the variance.

\begin{proof}[Proof of Theorem~\ref{thm:elliptic3dstatement}]
Given $r\in(0,1)$, we apply an ambient isometry $h:\mathbb{R}^{3} \to \mathbb{R}^{3}$ so that $U_r := h(U)$ has $\partial U_r$ contained in 
\[ \{ x\in \mathbb{R}^{3} : 1\leq |x - (r,0,0) | \leq 1+\epsilon \}, \]
that $(0,0,0) \in U_r$ is our point of interest, and that $\sigma_{0} = (\phi , \theta) = (0,0) \in S^{2}$.
Suppose that $\epsilon >0$.

As in the proof of Theorem~\ref{thm:ellipticstatement}, we will estimate $d_{U_r}$ by replacing $U_r$ with $V_r$ given by the analogue of Figure~\ref{fig:5}. For $x_{1}\geq 0$ we have $(x_{1},x_{2},x_{3}) \in \partial V_r$ if and only if $\sqrt{x_{1}^{2} + x_{2}^{2} +x_{3}^{2}} = 1 + \epsilon$, and for $x_{1} <0$, we have $(x_{1},x_{2},x_{3}) \in \partial V_r$ if and only if $ \sqrt{x_{1}^{2}+x_{2}^{2}+x_{3}^{2}} = 1$.

Recalling $G(\phi,\theta)$ from \eqref{eq:gpt}, the final term can be bounded via
\[ | d_{U_r}(0,\phi,\theta) - d_{U_r}(0,\phi - \pi , \theta) | \leq \operatorname{diam} U_r = \operatorname{diam} U.\]
Then by Theorem~\ref{thm:vrep3d}, we have
\begin{align*}
D_{\sigma_{0}} &v_{U_r}(0) = \frac{1}{4\pi} \int_{0}^{2\pi} \int_{0}^{\pi/2} \sin (2\phi ) G(\phi , \theta) \: d\phi \: d\theta \\
& \geq \frac{I_{1}^{U_r}(0)}{2\pi} \int_{0}^{2\pi} \int_{0}^{\pi/2} \sin(2\phi) \ln \left( \frac{d_{U_r}(0,\phi - \pi , \theta) }{d_{U_r}(0,\phi,\theta)} \right) \: d\phi \: d\theta - \operatorname{diam} U\\
& \geq \frac{I_{1}^{U_r}(0)}{2\pi} \int_{0}^{2\pi} \int_{0}^{\pi/2} \sin(2\phi) \ln \left( \frac{d_{V_r}(0,\phi - \pi , \theta) }{d_{V_r}(0,\phi,\theta)} \right) \: d\phi \: d\theta - \operatorname{diam} U\\
&= \frac{I_{1}^{U_r}(0)}{2\pi}  \int_{0}^{2\pi} \int_{0}^{\pi/2} \sin(2\phi) 
\ln \left( \frac{ -r\cos ( \phi - \pi) + \sqrt{ 1-r^{2} \sin^{2}(\phi - \pi) } } {-r\cos \phi + \sqrt{1-r^{2}\sin^{2}\phi} } \right) \: d\phi \: d\theta - \operatorname{diam} U \\
&= I_{1}^{U_r}(0)  \int_{0}^{\pi/2} \sin(2\phi) 
\ln \left( \frac{ r\cos ( \phi ) + \sqrt{ 1-r^{2} \sin^{2}(\phi ) } } {-r\cos \phi + \sqrt{1-r^{2}\sin^{2}\phi} } \right) \: d\phi  - \operatorname{diam} U
\end{align*}
To make further progress, we first deal with the following integral by substituting $u = r\sin \phi$ and then integrating by parts with $s = \ln ( \sqrt{r^{2}-u^{2}} + \sqrt{1-u^{2}})$ and $dt = 2ur^{-2} du$:
\begin{align*}
\int_{0}^{\pi/2} \sin(2\phi) &\ln \left( r\cos \phi + \sqrt{1-r^{2}\sin^{2}\phi } \right) \: d\phi \\
&= \int_{0}^{\pi/2} 2\sin \phi \cos \phi \ln \left( r\cos \phi + \sqrt{1-r^{2}\sin^{2}\phi } \right) \: d\phi \\
&= \int_{0}^r \frac{2u}{r^{2}} \ln \left( \sqrt{r^{2}-u^{2}} + \sqrt{1-u^{2}} \right) \: du\\
&= \left [ \frac{u^{2}}{r^{2}} \ln \left( \sqrt{r^{2}-u^{2}} + \sqrt{1-u^{2}} \right) \right ]_{u=0}^r + \int_{0}^r \frac{u^{3}}{r^{2}\sqrt{r^{2}-u^{2}}\sqrt{1-u^{2}} } \: du\\
&= \ln \sqrt{1-r^{2}} + \frac{1}{2r^{2}} \left( \frac{r^{2}+1}{2} \tanh^{-1}(r) - r \right).
\end{align*}
By an analogous computation that we omit, we have
\begin{align*} \int_{0}^{\pi/2} \sin(2\phi) \ln & \left( -r\cos \phi  + \sqrt{(1+\epsilon)^{2} - r^{2}\sin^{2}\phi} \right) \: d\phi \\
&= \ln \sqrt{ (1+\epsilon)^{2} - r^{2}} - \frac{1}{2r^{2}} \left( [( 1+\epsilon)^{2} + r^{2} ] \tanh^{-1} \left( \frac{r}{1+\epsilon} \right) - (1+\epsilon)r \right).
\end{align*}
Putting this altogether, we obtain
\begin{align*}
D_{\sigma_{0}} v_{U_r}(0) &\geq I_{1}^{U_r}(0) \left [ \frac{1}{2} \ln \left( \frac{ 1-r^{2}}{(1+\epsilon)^{2} - r^{2} }\right) \right . \\
& \hskip0.2in + \frac{1}{2r^{2}} \left( \left . (r^{2}+1) \tanh^{-1} (r) + ((1+\epsilon)^{2}+r^{2}) \tanh^{-1} (r/(1+\epsilon)) - \frac{ 2+\epsilon}{2r} \right) \right ] \\
& \hskip0.2in - \operatorname{diam} (U).
\end{align*}
We now fix $\epsilon >0$ and let $r\to 1^-$. 
We have
\begin{align*} 
\frac{1}{2} \ln (1-r^{2}) + \frac{r^{2}+1}{2r^{2}} \tanh^{-1}(r) &= \frac{1}{2} \ln(1-r^{2})+ \left( \frac{r^{2}+1}{4r^{2}} \right) \ln \frac{1+r}{1-r} \\
&= \left( \frac{1}{2} + \frac{r^{2}+1}{4r^{2}} \right) \ln (1+r) + \left( \frac{1}{2} - \frac{r^{2}+1}{4r^{2}} \right) \ln (1-r) \\
&= \left( \frac{1}{2} + \frac{r^{2}+1}{4r^{2}} \right) \ln (1+r) + \frac{r^{2}-1}{4r^{2}} \ln (1-r)\\
& \to \ln 2
\end{align*}
as $r\to 1^-$. Set $C = \sup_{r\in(0,1)} I_{1}^{U_r}(0) <\infty$. 
Therefore, we have
\begin{align*}
\lim_{r\to 1^-} D_{\sigma_{0}} v_{U_r} (0) &\geq C \left( \ln 2 -1  - \epsilon /2- \frac{1}{2} \ln (2\epsilon +\epsilon^{2}) + (1 + \epsilon + \epsilon^{2}/2) \tanh^{-1}(1/(1+\epsilon) ) \right) \\
& \hskip0.5in - \operatorname{diam} (U) \\
& = O \left( \frac{1}{\epsilon} \right).
\end{align*}
In particular, we conclude that there exists $\epsilon_{0}>0$ so that if $\epsilon \in (0,\epsilon_{0})$ then there exists $r(\epsilon) \in (0,1)$ so that if $r\in (r(\epsilon),1)$, then 
\[  D_{\sigma_{0}} v_{U_r} (0) > 0.\]
Undoing the ambient isometry from the start of the proof, and applying Theorem~\ref{thm:props}, completes the proof.
\end{proof}

Our estimates here are not optimal. In particular, we have not obtained that $r(\epsilon) \to 0$ as $\epsilon \to 0$. However, the value of Theorem~\ref{thm:elliptic3dstatement} is that it gives a uniform estimate of the type from Theorem~\ref{thm:vvsk3dstatement} that is valid in all directions.

\section*{Acknowledgments}

The authors would like to thank Stanley Strawbridge and Yunier Bello Cruz for interesting conversations on the topics of this paper.

\bibliographystyle{siamplain}
\bibliography{references}

\end{document}